\documentclass[a4paper,11pt]{article}     
\usepackage{amsmath,amssymb,amsfonts,amsthm} 
\usepackage{graphicx}                 
\usepackage{color}                    
\usepackage{hyperref}                 
\usepackage[nottoc, notlof, notlot]{tocbibind}  
\usepackage{cite}

\def\M{\mathcal{ M}}
\def\G{\mathcal{ G}}
\def\R{\mathbb{ R}}
\def\E{\mathbb{ E}}
\def\P{\mathbb{ P}}
\def\C{\mathcal{ C}}
\def\Z{\mathbb{ Z}}
\newtheorem{theorem}{Theorem}[section]
\newtheorem{proposition}[theorem]{Proposition}

\newtheorem{lemma}[theorem]{Lemma}
\newtheorem{remark}[theorem]{Remark}
\newtheorem{definition}[theorem]{Definition}

\date{}
\title{A Russo Seymour Welsh Theorem  for \\ critical site percolation on \(\Z^2\)}

\author{Xiaolin ZENG}

\begin{document}
\maketitle


\begin{abstract}
The Russo-Seymour-Welsh Theorem for \(\Z^2\) bond or \(\mathbb{T}\) (triangular lattice) site percolation states that at criticality, for all fixed real \(\lambda\), the probability of the existence of a horizontal occupied crossing of each rectangle with size \(n\times \lambda n\) is not degenerated when \(n\) tends to infinity. Turning to site percolation on \(\Z^2\), where the self duality does not hold anymore, we prove that the analogue statement of the RSW Theorem will still be true in this case.

The proof uses appropriate finite size criteria and a result of Kesten which allows us to extend existing crossings without losing too much probability. As a consequence, there is no infinite cluster at criticalty.

Our object in this short paper is twofold. Our first goal is to give a proof of a RSW Theorem for \(\Z^2\) site percolation.
Since the proof uses in an essential way a celebrated result by Kesten on the so called "box crossing", our second goal in this paper is to present a self-contained proof of this theorem in perhaps a more accessible language.
\end{abstract}

\section{Introduction and statement}
The Russo Seymour Welsh(RSW) box-crossing Theorem plays a key role in the theory of two-dimensional percolation, for example, it is used in the famous proof of Cardy's formula on triangular lattice\cite{smirnov2001critical}.

For site percolation on triangular lattice and bond percolation on square lattice, there is a nice proof involving planar duality. For other models where there is no such helpful property, one has to deal with them in different ways. A remarkable result of G.Grimmett and I.Manolescu shows that this remains true for a rich family of models~\cite{grimmett2012bond}.

The current dissertation gives a proof of the following RSW Theorem for site percolation on the square lattice.
The proof follows a scheme which was suggested to us by Hugo Duminil-Copin.


\begin{theorem}[RSW Theorem for critical {\bf site} percolation on $\Z^2$]\label{th.main}
  For any \(\lambda >0\), there exists \(c=c(\lambda)>0\), such that for all \(n\geq 1\)  
  
\begin{equation*}
c\leq \P_{p_c}[\text{ there exists a horizontal crossing of any } n \text{ by } \lambda n \text{ box }]\leq 1-c \,,
\end{equation*}
where $\P_{p_c}$ stands for critical site percolation on $\Z^2$. 
\end{theorem}

%


The remainder of the paper is organised as follows, in section 2, we give a sketch of the strategy and some basic set-ups. Section 3 contains a proof of a particular case of the main theorem for fixed \(\lambda(=2)\). In section 4 we give a self-contained proof of Kesten's result (Theorem \ref{kesten2})  which can be found in \cite{kesten1982percolation}. In the last section we discuss some consequences of the RSW Theorem.

\noindent

\section{The strategy and basic setup.}
When there is no planar self duality, we have to deal with both lower and upper bound of crossing probability. Fortunately, by introducing a weaker type of duality, known under the name of {\it matching pair}, we can deal with both of them at the same time.

\begin{definition}
A {\it mosaic} \(\M\) is a graph imbedded in \(\R^2\) such that 
\begin{enumerate}
\item \(\M\) is planar and contains no loops.
\item All edges of \(\M\) have finite length. Every compact set of \(\R^d\) intersects only finitely many edges of \(\M\).
\item Each component of \(\R^2\setminus \M\) is bounded by a polygon of finite number of edges.
\end{enumerate}

Let \(F\) be a face of mosaic \(\M\). {\it Close-packing} \(F\) means adding an edge to \(\M\) between any pair of vertices on the perimeter on \(F\) which are not yet adjacent.
\end{definition}

\begin{definition}
  Let \(\M\) be a mosaic and \(\mathcal{F}\) a subset of its faces. The {\it matching pair} \( (\G,\G^*)\) of graphs based on \( (\M,\mathcal{F})\) is the following pair of graphs: \(\G\) is the graph obtained from \(\M\) by close-packing all faces in \(\mathcal{F}\). \(\G^*\) is the graph obtained from \(\M\) by close-packing all faces not in \(\mathcal{F}\).
\end{definition}

\begin{remark}
In the sequel we only use the following matching pair \( (\Z^2, \Z^{2,*}) \) based on  \((\Z^2,\emptyset)\), see Figure \ref{Figure 1}.

Note that in a (general) matching pair usually at least one of the graphs is not planar, here \(\Z^{2,*}\) is not planar.
\end{remark}

\begin{figure}[!h]
\centering
\includegraphics[width=.6\textwidth]{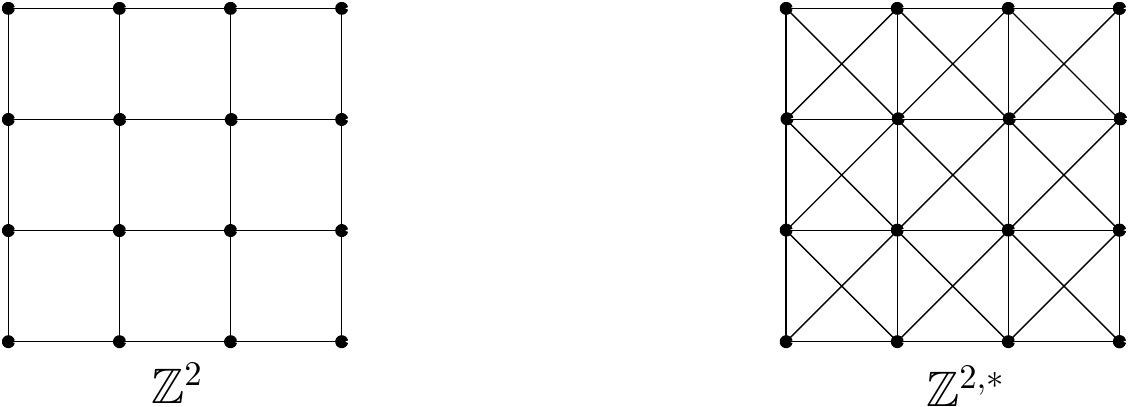}
\caption{{\it the matching pair \(\Z^2, \Z^{2,*}\).}}
\label{Figure 1}
\end{figure}

\begin{remark}
\label{rmk2.4}
The interest of introducing matching pair is the following: 

Consider a site percolation configuration in a rectangle \(R\) of \(\G\), define the dual configuration to be the set of those sites which are vacant and those edges with their endpoints both vacant. It follows that either there exists a horizontal occupied crossing of \(R\) only using edges of \(\G\), either there exists a vertical vacant crossing of \(R\) only using edges of \(\G^*\). See Figure \ref{Figure 1.2} for the case of the matching pair \( (\Z^2,\Z^{2,*})\).
\end{remark}

\begin{figure}[!h]
\centering
\includegraphics[width=.8\textwidth]{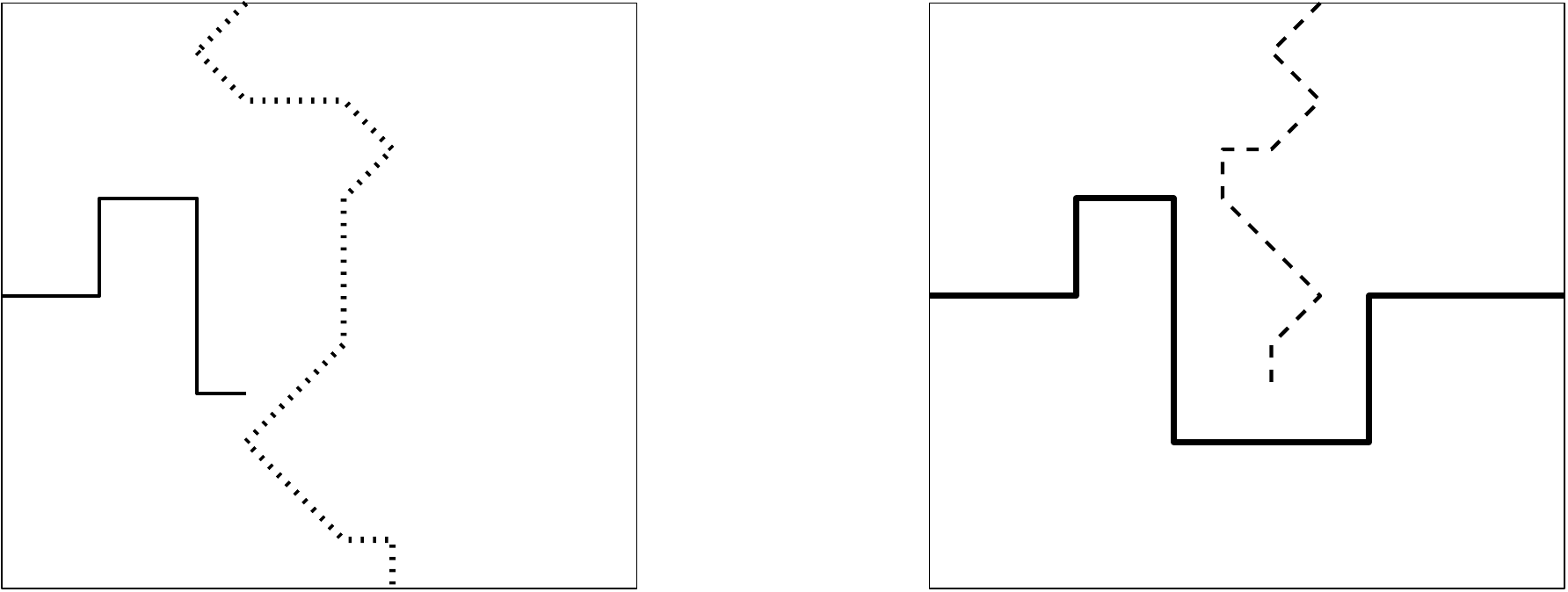}
\caption{{\it In the left box, a vertical crossing on \(\Z^{2,*}\) prohibits any horizontal crossing on \(\Z^2\), and vice versa in the right box.}}
\label{Figure 1.2}
\end{figure}

By remark \ref{rmk2.4} , it is equivalent to give an upper bound for crossing probability on \(\Z^2\) for \(p=p_c(\Z^2)\) or a lower bound for crossing probability but on \(\Z^{2,*}\) for \(p=p_c(\Z^{2,*})\).

\begin{remark}
The critical values of \(p\) on \(\Z^2\) and \(\Z^{2,*}\) satisfy \(p_c(\Z^2)+p_c(\Z^{2,*})=1\), this was shown in \cite{kesten1982percolation} chap 3.
\end{remark}

Here are some heuristics of the main proof:

As shown above, it is enough to give lower bounds on each graph of the matching pair \( \{ \Z^2, \Z^{2,*}\}\).
By introducing some finite size criteria, we show that the probability of the existence of a horizontal crossing in \([0,n]\times[0,2n]\) is larger than \(\frac{1}{25}\), both on \(\Z^2\) and on \(\Z^{2,*}\) at their respective critical points. Secondly, we state and prove a theorem of Kesten~(Theorem \ref{kesten2}) : one can extend existing crossing in such a way that the probability remains bounded from below on each graph of the matching pair.

Although the result of Kesten should be proved on each graph of the matching pair, it appears that one can deal with them at the same time, in the sequel, pictures are made on \(\Z^{2,*}\) in order to be general (as \(\Z^2\) is planar, it's clearer). In every proof we will point out what kinds of graph properties are needed, and show that both of the matching pair satisfy such properties.

\section{Finite criteria and consequences}
We will prove Theorem ~\ref{lem0} which will serve as a base of the RSW Theorem, this result can be compared to: {\it the probability that there exits a horizontal crossing of any \(n\) by \(n\) square is greater than \(\frac{1}{2}\)}, in the classic proof of the RSW Theorem for bond \(\Z^2\) percolation.

The statement of Theorem~\ref{lem0} is valid regardless whether we are on \(\Z^2\) or \(\Z^{2,*}\). Let \(\G\) denote either one of these graphs.

\begin{theorem}
\label{lem0}
  For site percolation on \(\G\), for all \(n\) large enough,
  \[\P_{p_c}[\text{ there exists a horizontal crossing of the box }[0,n]\times[0,2n]\text{ }]\geq \frac{1}{25},\] the critical value \(p_c\) is defined by \(p_c=\sup \{ p; \hspace{.2cm} \theta_{\G}(p)=0\}\), where \[\theta_{\G}(p)=\P_p[\text{the origin is contained in an infinite cluster}].\]
\end{theorem}


\subsection{First criterion}
Firstly, we consider a simpler version of finite volume
criterion of Aizenman and Newman(1984)~\cite{aizenman1984tree}.

\begin{proposition}
\label{prop2}
  Let \(p\in [0,1]\), if \(N_{p,\square_n}<1\), then there exists \(c>0\), such that for all \(a,b\), \(d(a,b)=m> n\), \(\P_p[a\leftrightarrow b]\leq e^{-cm}\).
\end{proposition}

Let us give some definitions and explain the meaning of \(N_{p,\square_n}\).

\begin{definition}
  Two sites \(a\) and \(b\) are connected, denoted \(a\leftrightarrow
  b\), if there exists an occupied path from \(a\) to \(b\). If \(A\) is a subset of \(\G\), then \(a\) is connected to some points of \(A\) is denoted by \(a\leftrightarrow A\). Define
  the connectivity function of \(a\) and \(b\), denoted \(\tau_p(a,b)\), to be \(\P_p[a\leftrightarrow b]\).
\end{definition}

Let \(\square_n\) be the square \([-n,n]\times[-n,n]\). We can define
the inner and outer boundary of \(\square_n\) (respectively denoted  \(\partial_-\square_n, \partial_+\square_n\))  by
\begin{itemize}
\item \(\partial_-\square_n=\{ a\in \square_n: \text{ }\exists b\notin
  \square_n\text{, }a\sim b\}, \)
\item \(\partial_+\square_n=\{ b\notin \square_n: \text{ } \exists a\in
  \square_n\text{, }b\sim a\}, \)
\end{itemize} 
where \(a\sim b\) means that \(a\) is adjacent to \(b\) in \(\G\).

Analogously, let \(\square_n(a)=\square_n+a\), and \(\tau_p(a,A)=\P_p[\exists a'\in A, a\leftrightarrow a']\) when \(A\) is subset of
\(\G\), let \(\tau_{p,A}(a\leftrightarrow b)=\P_{p,A}(a\leftrightarrow b)\) where \(\P_{p,A}\) denotes the probability measure restricted in \(A\).

Let \(N_{\square_n}\) be the value \(\vert \partial_+\square_n\vert 
\tau_{p,\square_n}(0,\partial_-\square_n)\). This value is an upper bound of the expected number of sites on the inner boundary of \(\square_n\), which are connected to
the origin by an occupied path inside \(\square_n\).

\begin{figure}[!h]
\centering
\includegraphics[width=.3\textwidth]{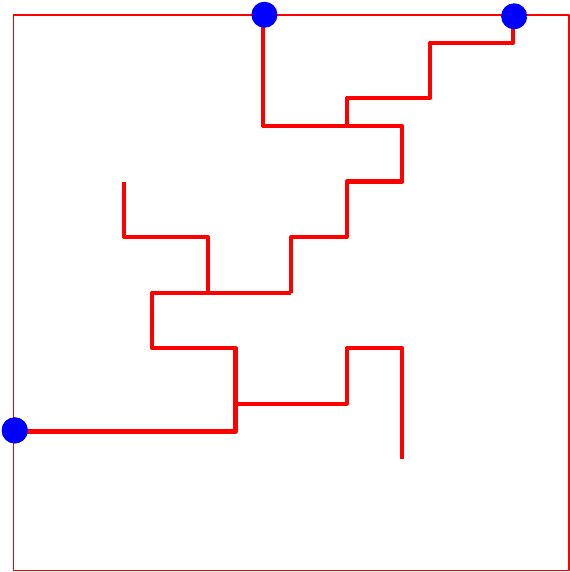}
\caption{{\it The number of sites on the boundary connected to the origin inside the square is 3.}}
\label{Figure 2}
\end{figure}

\begin{proof}

Let \(a\) be a site of \(\G\) far from the origin, more precisely \(d(0,a)=m>n\), we will show that
\(\tau_p(0,a)\leq e^{-cm}\) for some constant \(c\).

Recall that when \(0\leftrightarrow a\) occurs, there exists a path from \(0\) to
\(\partial_-\square_n\) and another disjoint path from
\(\partial_+\square_n\) to \(a\) as shown in Figure \ref{Figure 3} and \ref{Figure 4}.

\begin{figure}[!h]
\centering
\includegraphics[width=.65\textwidth]{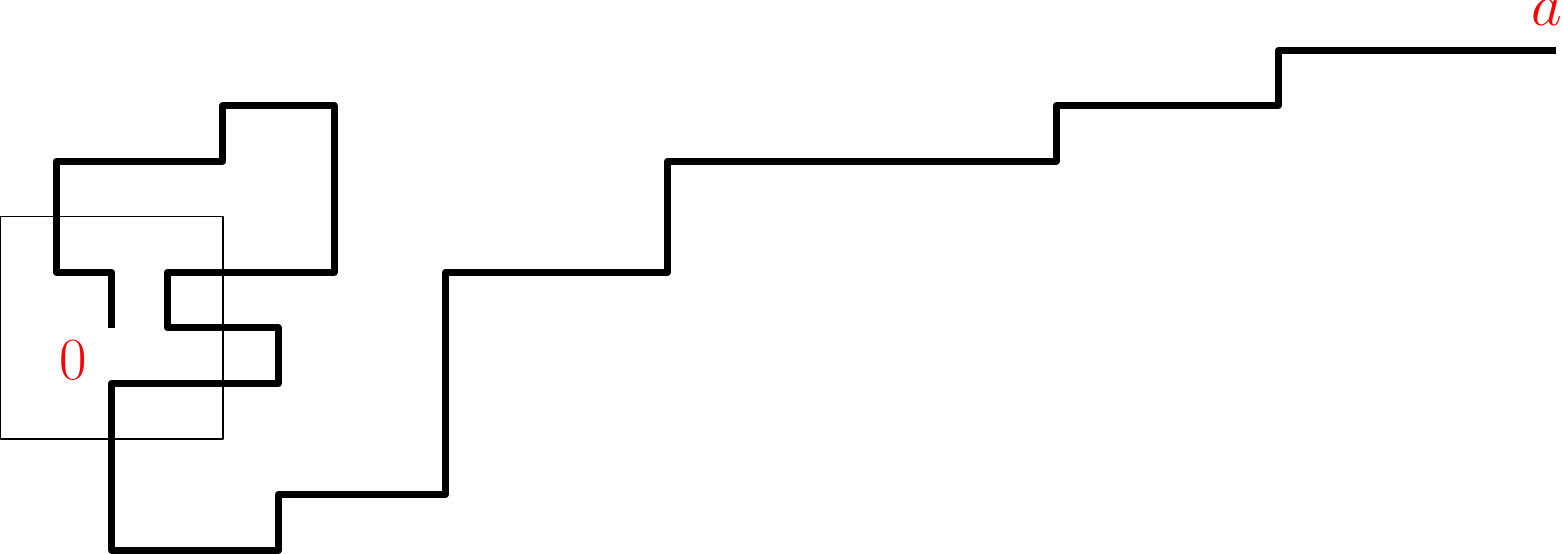}
\caption{{\it an occupied path from \(0\) to \(a\).}}
\label{Figure 3}
\end{figure}

\begin{figure}[!h]
\centering
\includegraphics[width=.65\textwidth]{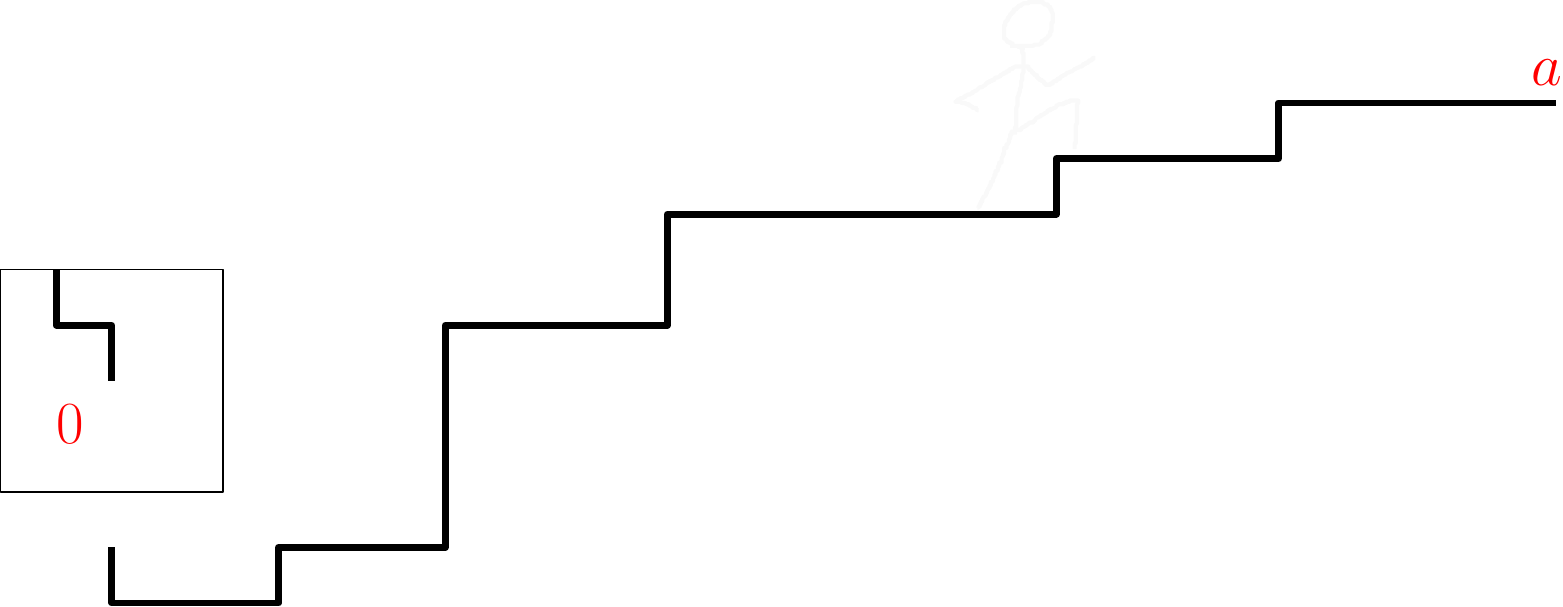}
\caption{{\it two occupied disjoint path \(0\leftrightarrow \partial_-\square_n\text{ inside
}\square_n\text{ ,}\partial_+\square_n\leftrightarrow a\text{ outside }\square_n\).}}
\label{Figure 4}
\end{figure}

Hence
\begin{align*}
  \tau_p(0,a)&=\P_p(0\leftrightarrow a)\\
&\leq \P_p(0\leftrightarrow \partial_-\square_n\text{ inside
}\square_n\text{ ,}\partial_+\square_n\leftrightarrow a \text{ outside }\square_n).
\end{align*}

As the two paths mentioned above belong to disjoint subsets of
\(\G\), it follows by independence that the last term is just a product.

Therefore 
\begin{align*}
\tau_p(0,a)&\leq\P_p[0\leftrightarrow \partial_-\square_n\text{ inside
}\square_n]\P_p[\partial_+\square_n\leftrightarrow a\text{ outside }\square_n]\\
&\leq\tau_{p,\square_n}(0,\partial_-\square_n)\tau_{p,\square_n^c}(\partial_+\square_n,a)\\
&\leq
\tau_{p,\square_n}(0,\partial_-\square_n)\tau_p(\partial_+\square_n,a)\\
&\leq
\tau_{p,\square_n}(0,\partial_-\square_n)\sum_{b\in \partial_+\square_n}\tau_p(b,a).
\end{align*}

Thus by introducing \(N_{\square_n}\), the above inequality can be written into \[\tau_p(0,a)\leq
\frac{N_{\square_n}}{\vert \partial_+\square_n\vert}\sum_{b\in \partial_+\square_n}\tau_p(b,a).\]

One can iterate the last inequality \(O(\lfloor \frac{m}{n}\rfloor\)) times,
therefore, \(\tau_p(0,a)\leq N_{\square_n}^{O(\lfloor m/n\rfloor)}\leq
e^{-mc}\).
\end{proof}

\begin{remark}
\label{rmk1}
  The reciprocal of Proposition ~\ref{prop2} is also true, i.e. if we have exponential decay of the
  connectivity function, then there exists \(n\) such that
  \(N_{\square_n}<1\). Indeed,
  \(\vert \partial_+\square_n\vert \asymp n\), 
  \(\tau_{p,\square_n}(0,\partial_-\square_n)\leq \tau_p(0,a)\) when
  \(d(0,a)=n\), and taking \(n\) large enough yields \(N_{\square_n}<1\).
\end{remark}

\subsection{Second criterion}
Let \(\C_h([n_1,m_1],[n_2,m_2])\)\footnote{For convinience, we also write \(\C_h(l_1,l_2)\) for \(\C_h([0,l_1],[0,l_2])\).} be the event that {\it there exists a horizontal
crossing of the rectangle \([n_1,m_1]\times [n_2,m_2]\)}. That is, an occupied
circuit inside the rectangle which connects the left and the right side of the
rectangle. Respectively let \(\C_v([n_1,m_1],[n_2,m_2])\) be the event that {\it there exists a vertical
crossing of the rectangle \([n_1,m_1]\times [n_2,m_2]\)}.

Now we can state the following finite-size criterion:
\begin{proposition}
 \label{prop1}
  If for some integer \(n\), \(\P_p[\C_h([0,n],[0,2n])]\leq \epsilon<\frac{1}{25}\),
  then we have exponential decay of the connectivity function: there exists \(c>0\), such that
  \(\P_p[a\leftrightarrow b]=\tau_p(a,b)\leq
  e^{-dc}\), where \(d=d(a,b)> n\) is the graph distance of \(a\) and \(b\).
\end{proposition}

\begin{figure}[!h]
\centering
\label{Figure 5}
\includegraphics[width=.3\textwidth]{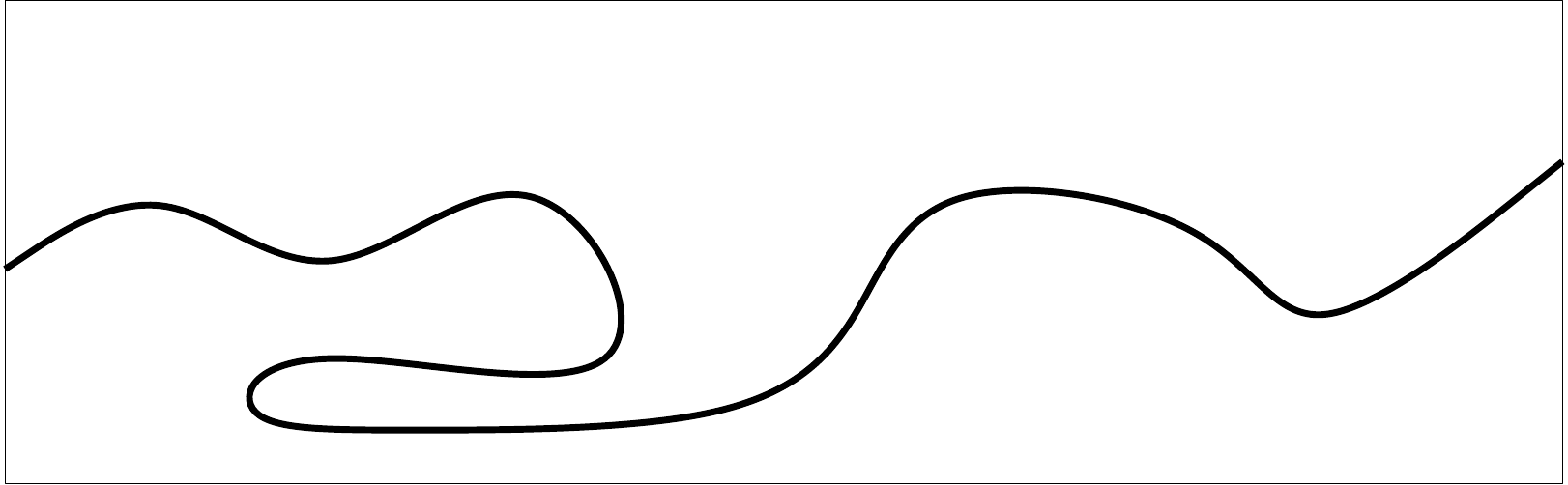}
\caption{{\it an occupied horizontal crossing.}}
\end{figure}

\begin{proof}
The proof is classic, see \cite{duminil2011near} Lemma 2.13 for a
more general case.

First we consider the rectangle \([0,n]\times[0,4n]\), and let
\(R_i,\text{ }i=1,...,5\) denote the following five rectangles: 
\begin{itemize}
\item \(R_1=[0,n]\times[0,2n]\)
\item \(R_2=[0,n]\times[n,3n]\)
\item \(R_3=[0,n]\times[2n,4n]\)
\item \(R_4=[0,2n]\times[n,2n]\)
\item \(R_5=[0,2n]\times[2n,3n]\)
\end{itemize}

As in Figure \ref{Figure 6}.

\begin{figure}[!h]
\centering
 \includegraphics[width=.2\textwidth]{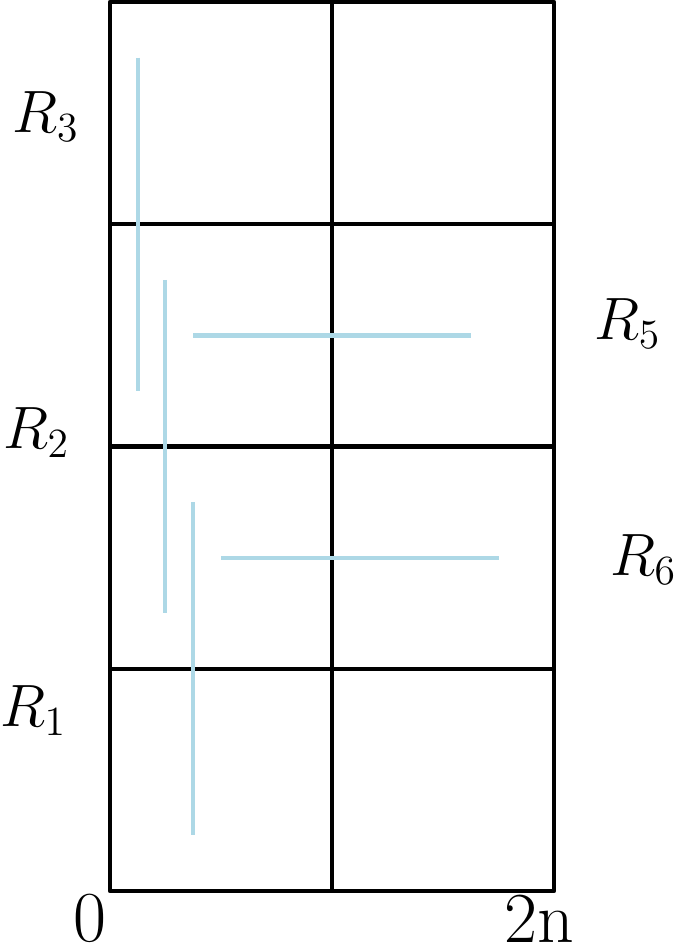} \hspace{1cm}
 \includegraphics[width=.2\textwidth]{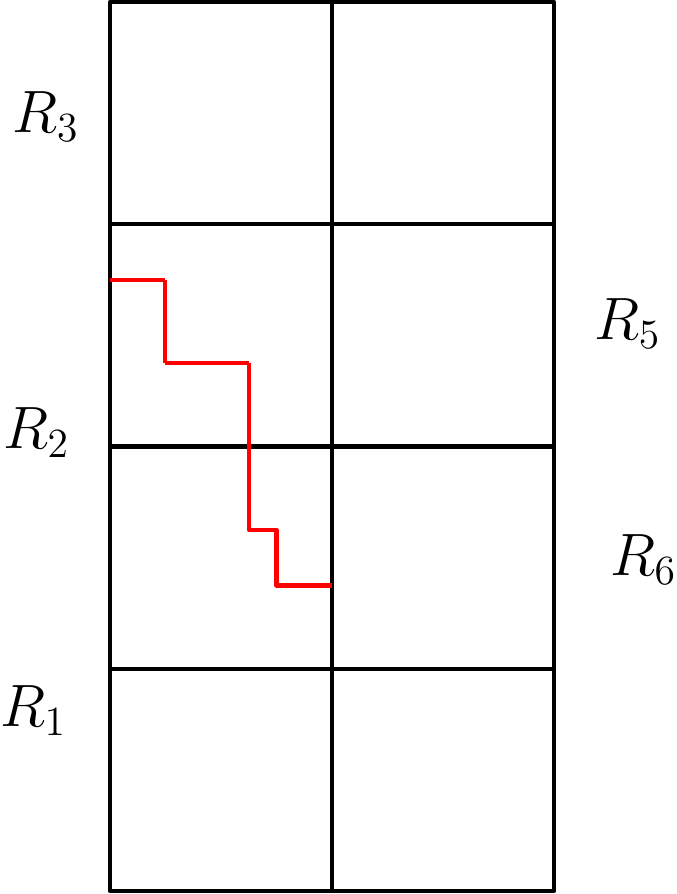}\hspace{1cm}
 \includegraphics[width=.2\textwidth]{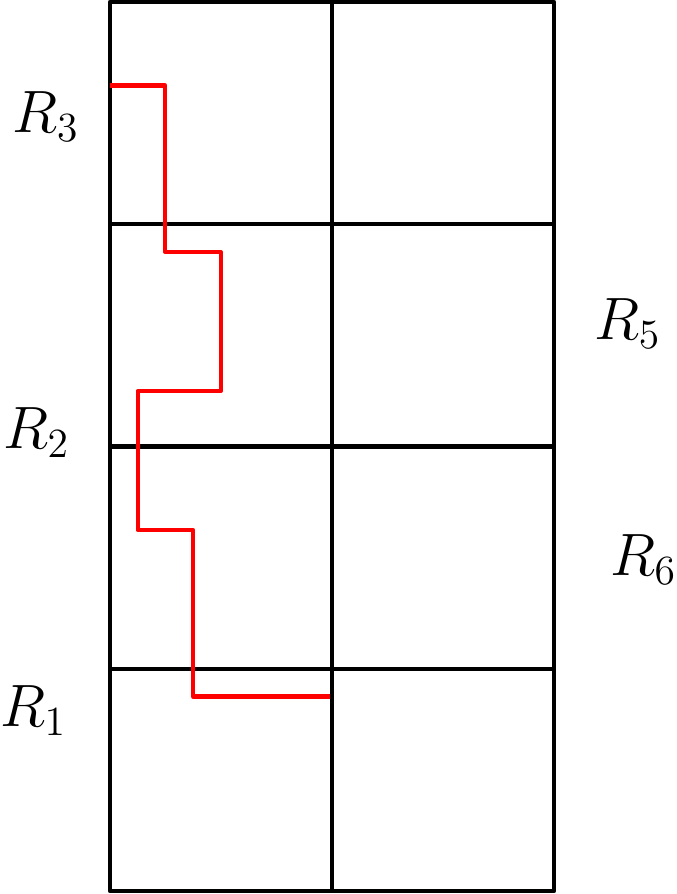}
\caption{{\it the first picture shows the five small rectangles, the remaining pictures give two examples in which the large one is crossed horizontally and thus at least one of the five small rectangle is crossed in its short direction.}}
\label{Figure 6}
\end{figure}

If \(\C_h([0,n],[0,4n])\) occurs, then at least one of the five
rectangles is crossed in its short direction, thus
\[\P_p[\C_h([0,n],[0,4n])]\leq \P_p[\cup_{i=1}^5 \{ R_i \text{ is crossed in its
  short direction}\}] \leq 5\P_p[\C_h([0,n],[0,2n])].\] 

Now as \(\C_h([0,2n],[0,4n])\) implies that both rectangles
\([0,n]\times[0,4n]\) and \([n,2n]\times[0,4n]\) are crossed in the
short direction, hence \(\P_p[\C_h([0,2n],[0,4n])]\leq \P_p[\C_h([0,n],[0,2n])]^2\) by
independence of the two crossings in disjoint rectangles.

By recurrence, \(25\P_p[\C_h([0,2^kn],[0,2^{k+1}n])]\leq
(25\P_p[\C_h([0,n],[0,2n])])^{2^k}\). Recall that \(\epsilon<\frac{1}{25}\), next  
\begin{equation}
  \label{eq:1}
  25\P_p[\C_h([0,2^kn],[0,2^{k+1}n])]\leq (25\epsilon)^{2^k} = e^{2^k \log
    (25\epsilon)} =e^{-c2^kn}
\end{equation}

with \(c=-\log(25\epsilon)/n>0\).

\begin{figure}[!h]
\centering
\includegraphics[width=.3\textwidth]{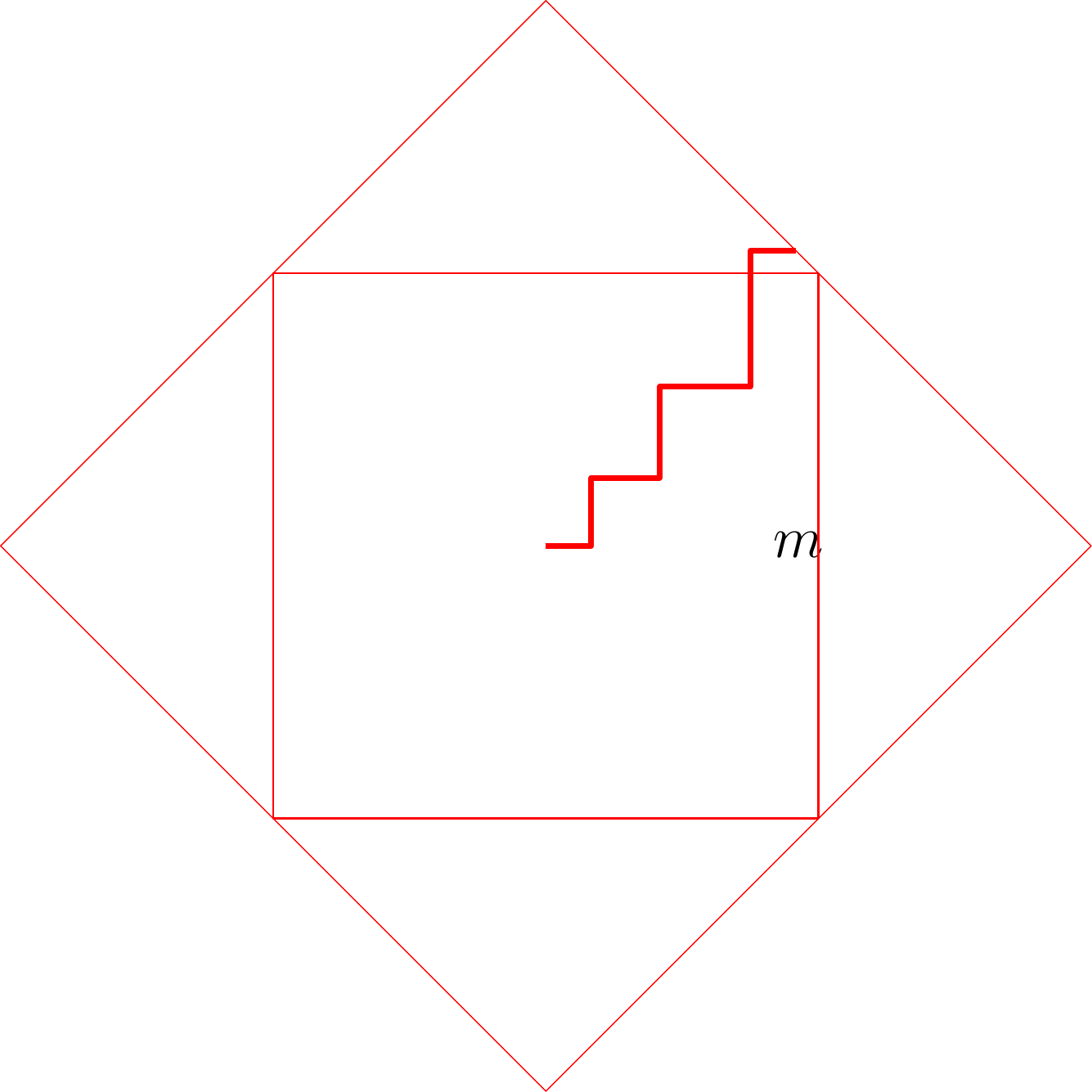}\hspace{4cm}
\includegraphics[width=.3\textwidth]{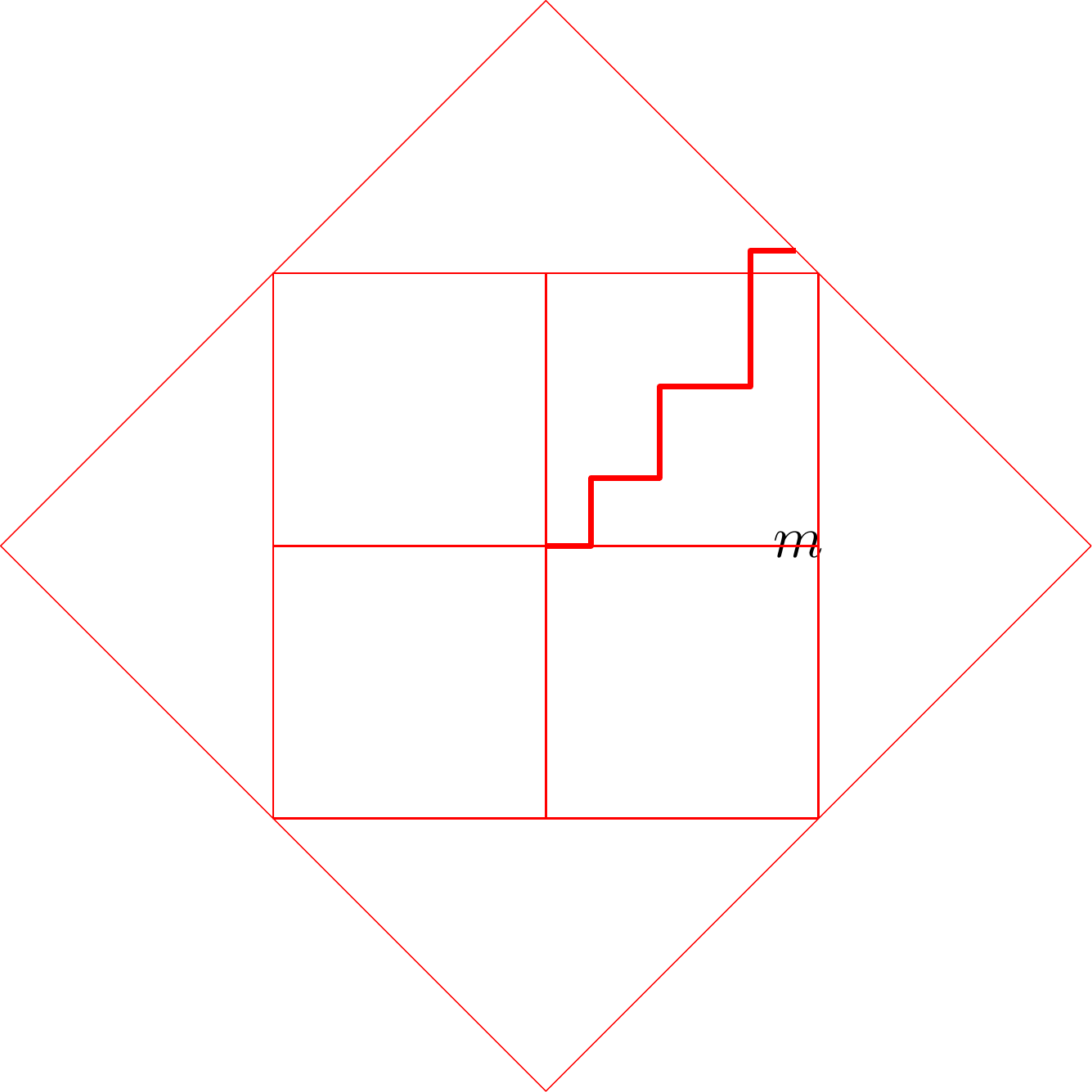}
\caption{\it the fact that the origin \(0\) is connected to some point at distance \(2m\) implies a crossing of one of four rectangles.}
\label{Figure 7}
\end{figure}

Let  \(A_m\) denote the event that {\it the origin is connected to
some point of graph distance \(m\) to the origin}, If \(A_{2m}\) occurs, then at least one of the following four rectangles: 
\([0,m]\times [-m,m]\), \([-m,0]\times[-m,m]\), \([-m,m]\times[0,m]\),
\([-m,m]\times[-m,0]\) is crossed
in the short direction, as shown in Figure \ref{Figure 7}.

Therefore, there exists  \(c>0\), such that
\begin{equation}
  \label{eq:2}
  \P_p[A_{2m}]\leq 4\P_p[\C_h([0,m],[0,2m])]\leq e^{-cN},
\end{equation}
which states the exponential decay of the radius of the cluster containing the origin.
\end{proof}

\subsection{Proof of Theorem \ref{lem0}}


\begin{proof}

Suppose by contradiction that there exists \(n\) such that \(
  \P_{p_c}[\C_h(n,2n)]<\frac{1}{25}\). By Proposition \ref{prop1}, there exists \(c>0\) such that
  \(\tau_p(0,\partial B_n)\leq e^{-nc}\). Next Remark \ref{rmk1} implies that there exists
  \(n_1\) such that \(N_{\square_{n_1}}<1\).

Let \(f_{n_1}: p\mapsto \E_p[N_{\square_{n_1}}]\), it is easy to check that
\(f_{n_1}\) is continuous and \( f_{n_1}(p_c)<1 \). 

By continuity, there exists \(\delta>0\) such that \(f_{n_1}(p_c+\delta)<1\). Now
Proposition \ref{prop2} gives us exponential decay of the connectivity function,
namely \(\P_{p_c+\delta}[A_n]\leq e^{-n\tilde{c}}\).

As we are at super-critical, it is impossible to have exponential decay.

Therefore, \(\P_{p_c}[\C_h(n,2n)]\geq \frac{1}{25}\).
\end{proof}

\section{Kesten's black box}

\subsection{Statement of the proposition}
It has been shown in Harry Kesten's book ~\cite{kesten1982percolation}
Lemma 6.4 that under appropriate condition, if the crossing probabilities of certain rectangles in both the horizontal and the vertical direction are bounded away from zero, then so are the crossing probabilities for larger rectangles.

We only consider site percolation on \(\G\), in which case
the proof is more accessible but still highlights the main ideas.

For readers not familiar with these kinds of arguments, it might appear quite technical. Therefore we will give a sketch prior to the actual formal proof, which tries to convain the main ideas.

The following proof will be valid on both \(\Z^2\) and \(\Z^{2,*}\), which is a consequence of the following fact: consider \(\G(\text{ either }\Z^2\text{ or }\Z^{2,*})\) as graph embedded in \(\R^2\), given any two curves in \(\G\), if they intersect in \(\R^2\), then their union is connected in \(\G\). This is not always true for non planar graphs, but here \(\Z^{2,*}\) is particular, as shown in Figure \ref{ZstarI}.

\begin{figure}[!h]
   \centering
 \includegraphics[width=.7\textwidth]{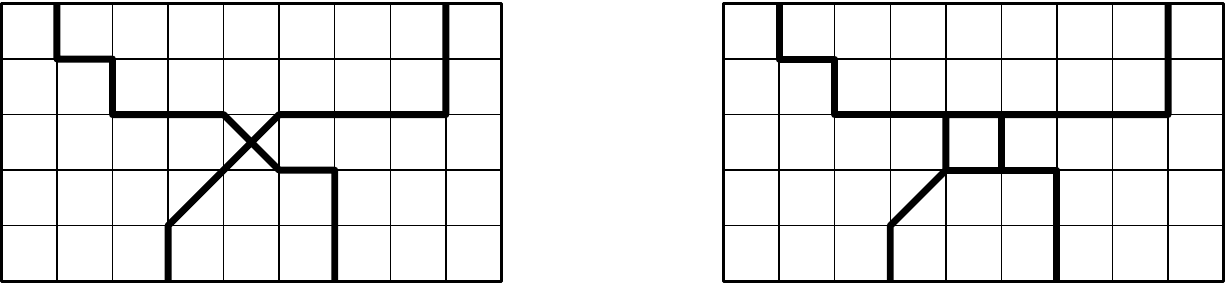}
 \caption{\it Two curves intersect at a point which is not a vertex of \(\Z^{2,*}\), but their union forms a connected component on \(\Z^{2,*}\).}
 \label{ZstarI}
\end{figure}

\begin{theorem}
  \label{kesten2}
Suppose that \(\P_p[\C_h(l_1,l_2)]\geq \delta_1 >0\) and
 \(\P_p[\C_v(l_3,l_2)]\geq \delta_2 >0\) for some integers
 \(l_1,l_2,l_3> 1\) with \(l_3\leq tl_1\) for
 some \(t\) and \(l_1\geq 48\).

Then for each \(k\) there exists a constant \(c(\delta_1,\delta_2,t,k)>0\)
such that \(\P_p[\C_h(kl_1,l_2)]\geq c>0\).
\end{theorem}

 \begin{remark}
   Although the proof only concerns with \(\Z^2\) and \(\Z^{2,*}\), the idea can be applied to more general cases. In fact, we have the more general theorem \ref{kesten3} (we only give the proof of \ref{kesten2}). 
\begin{theorem}
\label{kesten3}
For any matching pair \( (\G,\G^*) \) satisfying: 
   \begin{enumerate}
   \item \(\G\) is planar or the union of any two curves intersecting on \(\G\) forms a connected component.
   \item \(\G\) is periodic and the second coordinate axis is an axis of symmetry.
   \item The percolation is of finite types and the measure is symmetric with respect to the second coordinate axis.
   \item The length of any edge of \(\G\) is bounded uniformly by \(\Lambda\).
   \end{enumerate}

For all \(\delta_1>0\), \(\delta_2>0\), for all \(k\in \Z^+\), there exists \(c(\delta_1,\delta_2,k)>0\) such that if 
\[\P_p[\C_h([0,l_1],[0,l_2])]\geq \delta_1 >0\]
\[\P_p[\C_v([0,l_3],[0,l_2])]\geq \delta_2 >0\] for some integers \(l_1,l_2,l_3\geq 1\) with \(l_1\geq 32+16\Lambda\), \(l_2>\Lambda\) (here \(\P_p\) stands for the site percolation on \(\G\)); Then \(\P_p[\C_h(kl_1,l_2)]\geq c>0\), moreover, \(\lim_{\delta_1,\delta_2\rightarrow 1}c(\delta_1,\delta_2,k)=1\).
\end{theorem}
 \end{remark}

To prove Theorem \ref{kesten2}, we first prove the following weaker result:
\begin{proposition}
\label{kesten1}
 Suppose that \(\P_p[\C_h(l_1,l_2)]\geq \delta_1 >0\) and
 \(\P_p[\C_v(l_3,l_2)]\geq \delta_2 >0\) for some integers
 \(l_1,l_2,l_3\geq 1\) with \(l_3\leq \frac{7}{4}l_1, l_1\geq 48,
 l_2>1\).

 \begin{figure}[!h]
   \centering
     \includegraphics[width=.85\textwidth]{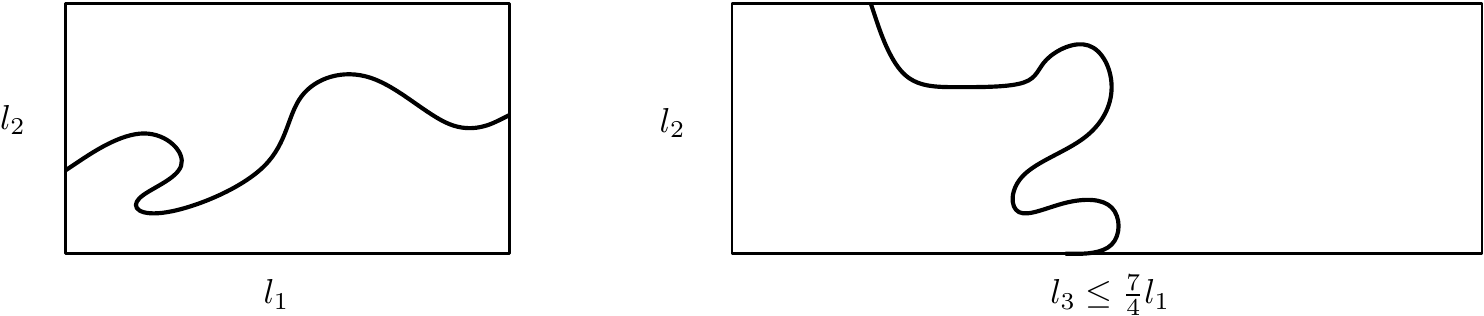}
 \end{figure}

Then for each \(k\) there exists \(c(\delta_1,\delta_2,k)>0\) such
that \(\P_p[\C_h(kl_1,l_2)]\geq c>0\).  
\end{proposition}

If we relax the condition \(l_3\leq\frac{7}{4}l_1\) in Proposition \ref{kesten1}
 we have Theorem \ref{kesten2}. Although we need Theorem \ref{kesten2} for our goal, the proof of 
 Proposition \ref{kesten1} already shows the most important argument.

Lemma \ref{lem1} and Lemma \ref{lll4} will serve to demonstrate Proposition \ref{kesten1} and Lemma \ref{lem2} allows us to relax the condition \(l_3\leq \frac{7}{4}l_1\).

\subsection{Step one: Lemma \ref{lem1}}

The main idea of the first step is to use Harris inequality to give a lower bound for certain event, we focus on the box
\([0,l_1]\times[0,l_2]\) and its reflection by the axis \(x=l_1\).

Let \begin{itemize} 
\item \(L_1=\{ x=\frac{l_1}{8} \}\) ,
\item \(L_2=\{ x=\frac{15l_1}{8} \} \).
\end{itemize}

Consider a fixed horizontal crossing \(r\)
of our rectangle, and it's mirror reflection \(\tilde{r}\) by the axis \(x=l_1\). Let \(r^-\) be the curve from the last intersection of \(r\) and \(L_1\) to the right edge of \([0,l_1]\times[0,l_2]\), also let \(\tilde{r}^-\) be the mirror reflection of \(r^-\).

And let \(B,T,L,R\) denote respectively the bottom, top, left and right side
of the entire rectangle \([0,2l_1]\times[0,l_2]\).

Let \(D(r)\) be the event that there exists a path \(s=v_0\rightarrow v_1\rightarrow \cdots \rightarrow v_z\) such that
\begin{itemize}
\item \(s\) connects \(r^-\) to \(T\),
\item \(s\) stays between \(L_1,L_2\) inside the rectangle and above the union of \(r^-\) and \(\tilde{r}^-\),
\item \(v_1,\cdots,v_z\) are occupied but \(v_0\) is not required to be occupied,
\end{itemize}
 as shown in Figure \ref{Figure 8}.

More formally,

\(D(r)=\{ \exists \text{ path } s=(v_0,...,v_z)\) such that 
\(v_1,...,v_z\) are occupied, \(v_0\) is not required to be occupied, \(v_0\rightarrow v_1\text{ intersects } r^-,v_z\in T,\) and \( \{
v_1,...,v_z\}\subset [\frac{l_1}{8},\frac{15l_1}{8}]\times[0,l_2]\)
and do not intersect \(r^-\cup\tilde{r}^- \} \).

\begin{figure}[!h]
\centering
\includegraphics[width=.8\textwidth]{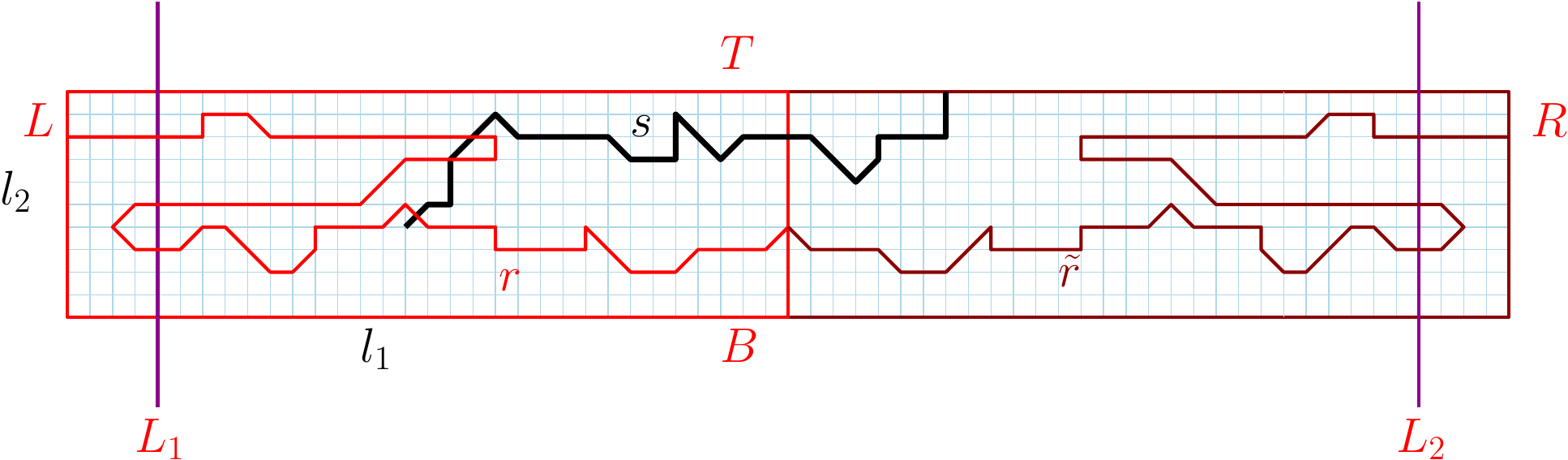}
\caption{\it The event \(D(r)\).}
\label{Figure 8}
\end{figure}

\begin{lemma}
\label{lem1}
For a fixed path \(r\),
  \begin{equation}
    \label{eq:3}
    \P_p[D(r)]\geq 1-\sqrt{1-\delta_2}.
  \end{equation}
\end{lemma}

\begin{proof}
  The main tool is the Harris inequality\footnote{ Some authors call it {\it the
  FKG inequality}}: when \(E_1,E_2\) are both increasing (or decreasing) events, one has
  \(\P[E_1\cap E_2]\geq \P[E_1]\P[E_2]\). This can also
  be rewritten as 
  \begin{equation}
    \label{eq:4}
    (1-\P[E_1])(1-\P[E_2])\leq 1-\P[E_1\cup E_2].
  \end{equation}

Let \(D(\tilde{r})\) be the symmetric event of \(D(r)\): {\it there exists
a path \(s'\) which is the same as \(s\) except that it connects
\(\tilde{r}^-\) to \(T\).}

\begin{figure}[!h]
\centering
\includegraphics[width=.8\textwidth]{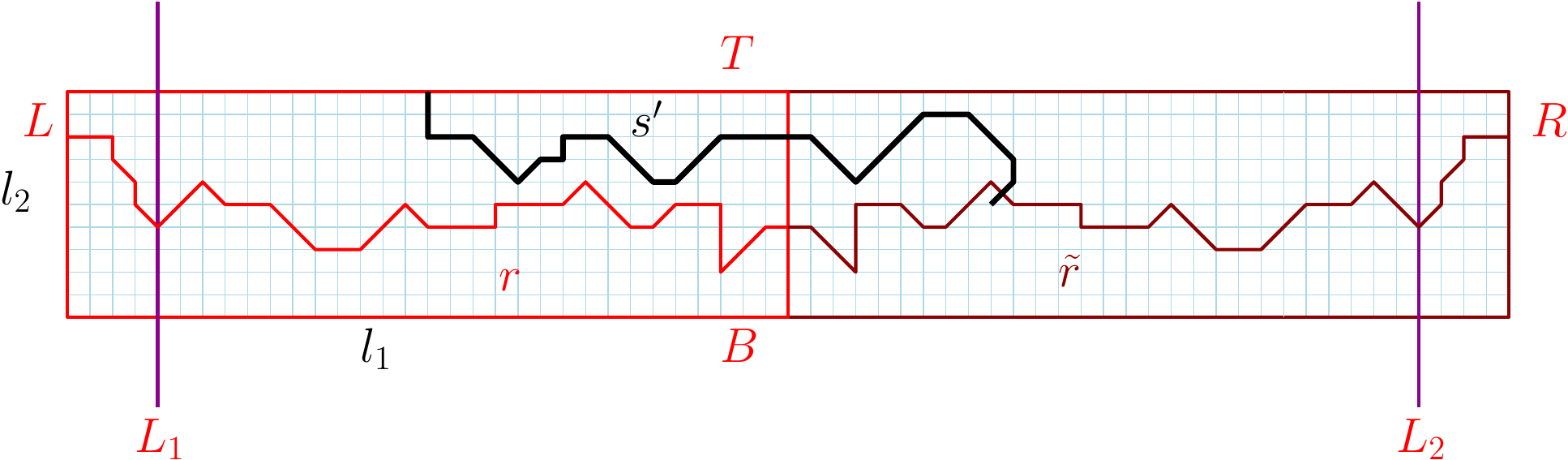}
\caption{\it The event \(D(\tilde{r})\).}
\label{Figure 9}
\end{figure}

By symmetry, it follows that \(\P_p[D(r)]=\P_p[D(\tilde{r})]\), applying \eqref{eq:4}
with \(E_1=D(r),E_2=D(\tilde{r})\), \[\P_p[D(r)]\geq
1-\sqrt{1-\P_p[D(r)\cup D(\tilde{r})]}.\]

To achieve our goal, it is enough to show that \(\P_p[D(r)\cup D(\tilde{r})]\geq \delta_2\).

To see this let us consider the rectangle
\([\frac{l_1}{8},\frac{l_1}{8}+l_3]\times[0,l_2]\) and assume that
there is a vertical crossing \(t\) of this rectangle, it is clear that
\(t\) will intersect \(r^-\cup\tilde{r}^-\) between \(L_1,L_2\). Hence
either \(D(r)\) or \(D(\tilde{r})\) occurs.

\begin{figure}[!h]
\label{Figure 10}
\centering
\includegraphics[width=.8\textwidth]{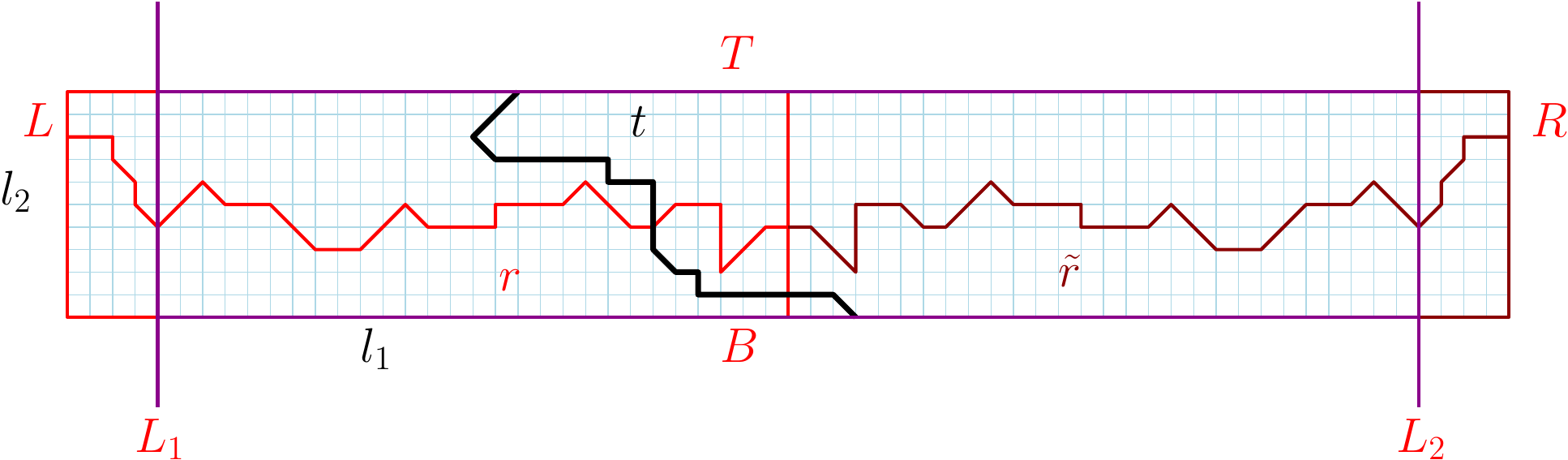}
\caption{\it A vertical crossing implies \(D(r)\) or \(D(\tilde{r})\).}
\end{figure}

Therefore, \[\P_p[D(r)]\geq
1-\sqrt{1-\P_p[D(r)\cup D(\tilde{r})]}\geq 1-\sqrt{1-\delta_2}.\]
\end{proof}

\subsection{Step two: Lemma \ref{lll4}}
Turning to the second step, an important remark on crossing events is
that if a rectangle is crossed horizontally, then there must be a
lowest (and a highest) crossing. One can find such a crossing using percolation interface
exploration.

If \(r\) is a fixed horizontal crossing of \([0,l_1]\times[0,l_2]\),
denote by \(Y(r)\) the second coordinate of the last intersection of
\(r\) with \(L_1\), as shown in Figure \ref{Figure 11}.

\begin{figure}[!h]
\centering
\includegraphics[width=.8\textwidth]{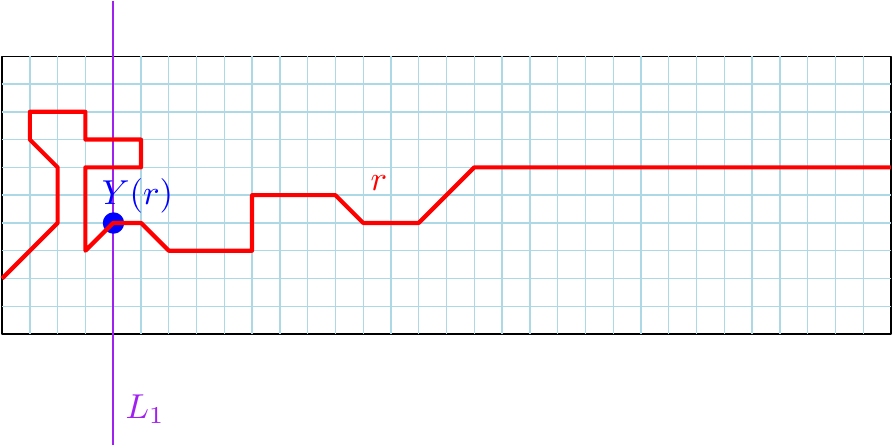} 
\caption{\it \((\frac{l_1}{8},Y(r))\) is the last point of the intersection of \(r\) and \(L_1\).}
\label{Figure 11}
\end{figure}

Let \(R^-\) be the lowest crossing of \([0,l_1]\times[0,l_2]\), let 
\(\epsilon=\frac{1}{\delta_1}(\sqrt{1-\delta_1}-(1-\delta_1))\) and
\(m\in \mathbb{R}\) be such that 
\begin{equation}
  \label{eq:5}
  \P_p[R^-\text{ exists and }Y(R^-)\leq m]\geq (1-\epsilon)\P_p[R^-\text{ exists}],
\end{equation}
\begin{equation}
  \label{eq:6}
  \P_p[R^-\text{ exists and }Y(R^-)< m]\leq (1-\epsilon)\P_p[R^-\text{ exists}].
\end{equation}

As the lattice space is discrete, it follows that such \(m\) is well defined.

As in Lemma \ref{lem1}, let \(T,B\) be the top and bottom edge of
\([0,2l_1]\times[0,l_2]\), and let
\(H=[\frac{l_1}{8},\frac{15l_1}{8}]\times[0,l_2]\).
 
Let \(E^-\) denote the event \(\cup_{r,Y(r)\leq m}\{ R^-=r\text{ and }D(r)\}\).

\begin{figure}[!h]
\label{Figure 12}
\centering
\includegraphics[width=.8\textwidth]{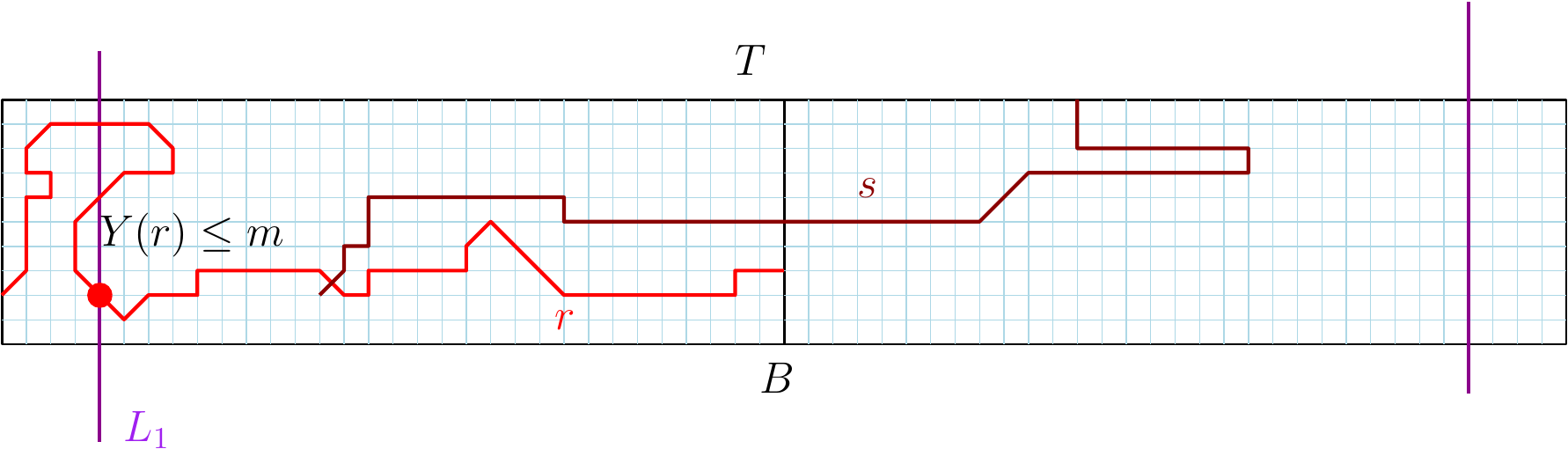} 
\caption{\it The event \(E^-\).}
\end{figure}

Respectively let \(E^+\) be the event \(\cup_{r,Y(r)\geq m}\{R^+=r\text{ and }D'(r)\}\), where \(D'(r)\) is the event that there exists a path \(s'\) which is the same as the path in the event \(D(r)\), except
 that it connects \(r^-\) to \(B\). Note that using the same argument as in Lemma \ref{lem1}, \(\P[D'(r)]\geq 1-\sqrt{1-\delta_2}\).

\begin{figure}[!h]
\label{Figure 13}
\centering
\includegraphics[width=.8\textwidth]{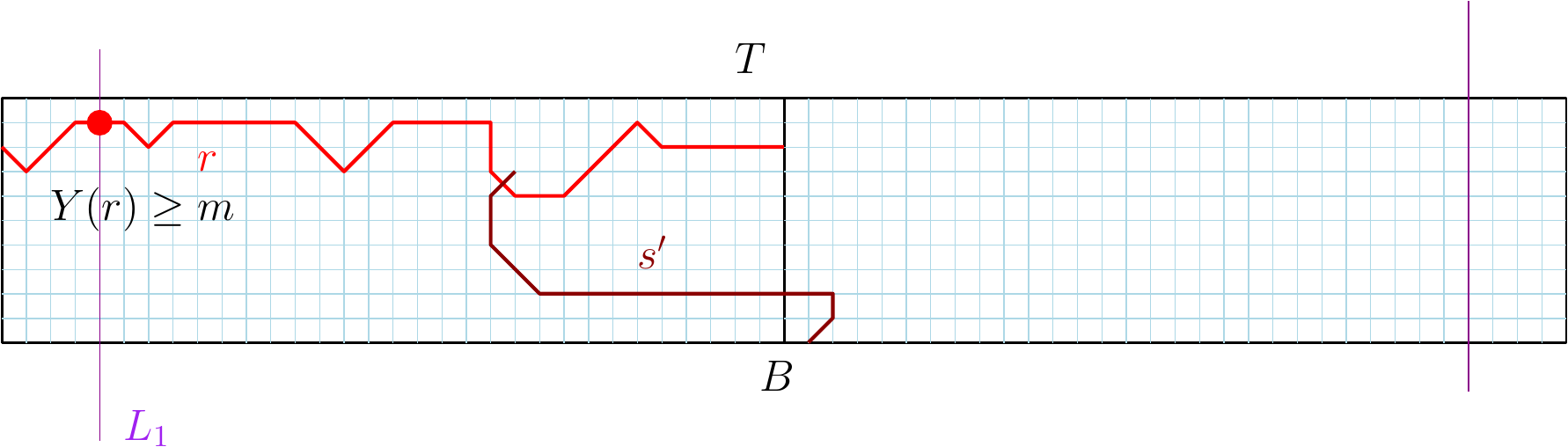}
\caption{\it The event \(E^+\).} 
\end{figure}

\begin{lemma}
\label{lll4}
  \begin{equation}
    \label{eq:7}
    \P_p[E^-]\geq (1-\sqrt{1-\delta_1})(1-\sqrt{1-\delta_2}),
  \end{equation}
  \begin{equation}
    \label{eq:8}
    \P_p[E^+]\geq (1-\sqrt{1-\delta_1})(1-\sqrt{1-\delta_2}).
  \end{equation}
\end{lemma}

\begin{proof}
Firstly we prove \eqref{eq:7}, recall that \(E^-=\cup_{r,Y(r)\leq m}\{R^-=r\text{ and }D(r)\}\), where the union is over all horizontal crossings
\(r\) of \([0,l_1]\times[0,l_2]\) with \(Y(r)\leq m\).

Note that all events in this union are disjoint.

Therefore, \[\P_p[E^-]= \sum_{r,Y(r)\leq m}\P_p[R^-=r]\P_p[D(r)\vert
R^-=r].\] If \(R^-=r\) occurs, by the independent property of percolation interface exploration, as \(D(r)\) only depends on the edges above \(R^-\), \(\P_p[D(r)\vert R^-=r]=\P_p[D(r)]\) , therefore, 
\begin{align*}
\sum_{r,Y(r)\leq m}\P_p[R^-=r]\P_p[D(r)\vert R^-=r]&=\sum_{r,Y(r)\leq m}\P_p[R^-=r]\P_p[D(r)]\\
&\geq (1-\sqrt{1-\delta_2})\P_p[R^- \text{ exists and }Y(R^-)\leq m]\\
&\geq (1-\sqrt{1-\delta_2})(1-\epsilon)\delta_1\\
&=(1-\sqrt{1-\delta_1})(1-\sqrt{1-\delta_2}).
\end{align*}

Turning to the proof of \eqref{eq:8}, we are going to use the same idea as in the proof of  \eqref{eq:7}, but
first we need to prove the following counterpart of \eqref{eq:5} :
\begin{equation}
  \label{eq:9}
  \P_p[R^+ \text{exists and }Y(R^+)\geq m]\geq (1-\epsilon)\delta_1,
\end{equation}
where \(R^+\) denotes the highest occupied horizontal crossing of
\([0,l_1]\times[0,l_2]\).

We have assumed that \(\P_p[R^- \text{ exists and } Y(R^-)<m]\leq (1-\epsilon)\P_p[R^- \text{ exists}]\),
in the sequel we are going to transfer this inequality into \eqref{eq:9}.

The idea to prove \eqref{eq:9} is shown in the diagram below.

\begin{figure}[!h]
  \centering
  \includegraphics[width=.8\textwidth]{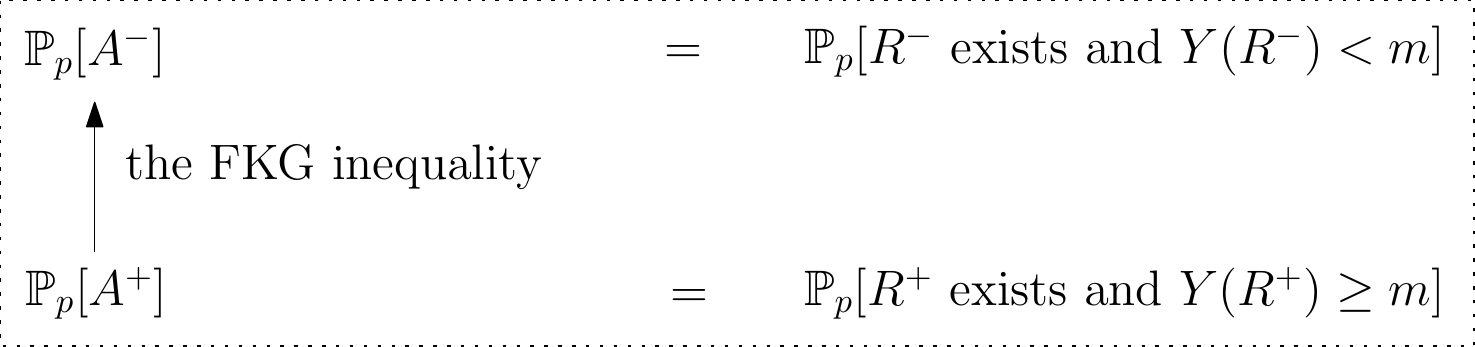}
\end{figure}

Firstly if there exists an occupied crossing \(r\) of \([0,l_1]\times[0,l_2]\), then
both \(R^-\) and \(R^+\) exist, moreover,
\begin{equation}
  \label{eq:10}
  Y(R^+)\geq Y(r)\geq Y(R^-).
\end{equation}

Let \(A^+ (\text{respectively }A^-)\) be the event \(\{\text{\it there exists an occupied horizontal crossing }r \text{\it\hspace{0.05cm} and } Y(r)\geq m\) \((\text{respectively }Y(r)<m)\}\), we have 

 \[\P_p[A^+]= \P_p[R^+ \text{ exists and } Y(R^+)\geq m],\]
 \[\P_p[A^-]= \P_p[R^- \text{ exists and } Y(R^-)< m].\]

As the events \(A^+\) and \(A^-\) are increasing and their union is the event \(\C_h(l_1,l_2)\) of probability \(\delta_1\),
we can apply the FKG inequality \eqref{eq:4} to this two increasing events, we have 

\begin{align*}
\P_p[R^+ \text{exists and }Y(R^+)\geq m]&=\P_p[A^+]\\
&\geq 1-\frac{1-\P_p[\C_h(l_1,l_2)]}{1-\P_p[A^-]}\\
&\geq 1-\frac{1-\delta_1}{1-\P_p[R^-\text{ exists and } Y(R^-)<m]}\\
&\geq 1-\frac{1-\delta_1}{1-(1-\epsilon)\delta_1}\\
&= (1-\epsilon)\delta_1.
\end{align*}

With \eqref{eq:9} we can reproduce the proof of \eqref{eq:7} to obtain \eqref{eq:8}, more precisely: 

\begin{align*}
\P_p[E^+]&=\sum_{r', Y(r')\geq m}\P_p[R^+=r']\P_p[D'(r')\vert R^+=r']\\
&\geq \sum_{r',Y(r')\geq m}\P_p[R^+=r']\P_p[D'(r')]\\
&\geq (1-\sqrt{1-\delta_2})\P_p[R^+ \text{ exists and } Y(R^+)\geq m]\\
&\geq (1-\sqrt{1-\delta_2})(1-\epsilon)\delta_1\\
&=(1-\sqrt{1-\delta_2})(1-\sqrt{1-\delta_1}).
\end{align*}

\end{proof}

\subsection{Step three: Proof of the proposition \ref{kesten1}}
Before the proof, we give some heuristics here, essentially we apply Lemma \ref{lll4}. Firstly we construct a horizontal crossing of length \(\frac{9}{8}l_1\), suppose that \(E^-\) occurs with \(R^-=r',\ s=s'\), \(E^+\) occurs with \(R^+=r'',\ s=s''\). Then \(r'\cup s'\) contains a continuous curve \(t'\) from \((\frac{l_1}{8},Y(r'))\) to the upper edge of \(H\), similarly \(r''\cup s''\) contains a continuous curve \(t''\) from \( (\frac{l_1}{8},Y(r''))\) to the lower edge of \(H\), as \(Y(r')\leq Y(r'')\), the union of \(t'\) and \(t''\) ensure a vertical crossing of \(H\). Thus by introducing another horizontal crossing of \([\frac{l_1}{8},\frac{9}{8}l_1]\times [0,l_2]\), the last will intersect the union of \(s'\) and \(s''\), therefore the union of these crossings will ensure the desired horizontal crossing of \([0,\frac{9}{8}l_1]\times [0,l_2]\) with positive probability.

The same argument can be applied to the box \([\frac{l_1}{8},\frac{9}{8}l_1]\times [0,l_2]\), this time we construct a horizontal crossing of length \(\frac{10}{8}l_1\) and by recurrence one can prolong as desired. 

\begin{proof}

Suppose that \(E^-,E^+\) both occur, thus \((r'\cup s'\cup r''\cup
s'')\cap H\) contains an occupied vertical crossing \(\psi\) of \(H\).

Let \(L_3=\{x=\frac{l_1}{8}+l_1\}
\), in order to have an occupied horizontal crossing of
\([0,\frac{l_1}{8}+l_1]\times[0,l_2]\), we distinguish two cases:

\begin{enumerate}
\item If \(\psi\) contains one point on or to the right of \(L_3\),
  then \(r'\cup s'\cup r''\cup s''\) contains an occupied horizontal
  crossing of \([0,\frac{l_1}{8}+l_1]\times[0,l_2]\).

\begin{figure}[!h]
\label{Figure 14}
\centering
\includegraphics[width=.8\textwidth]{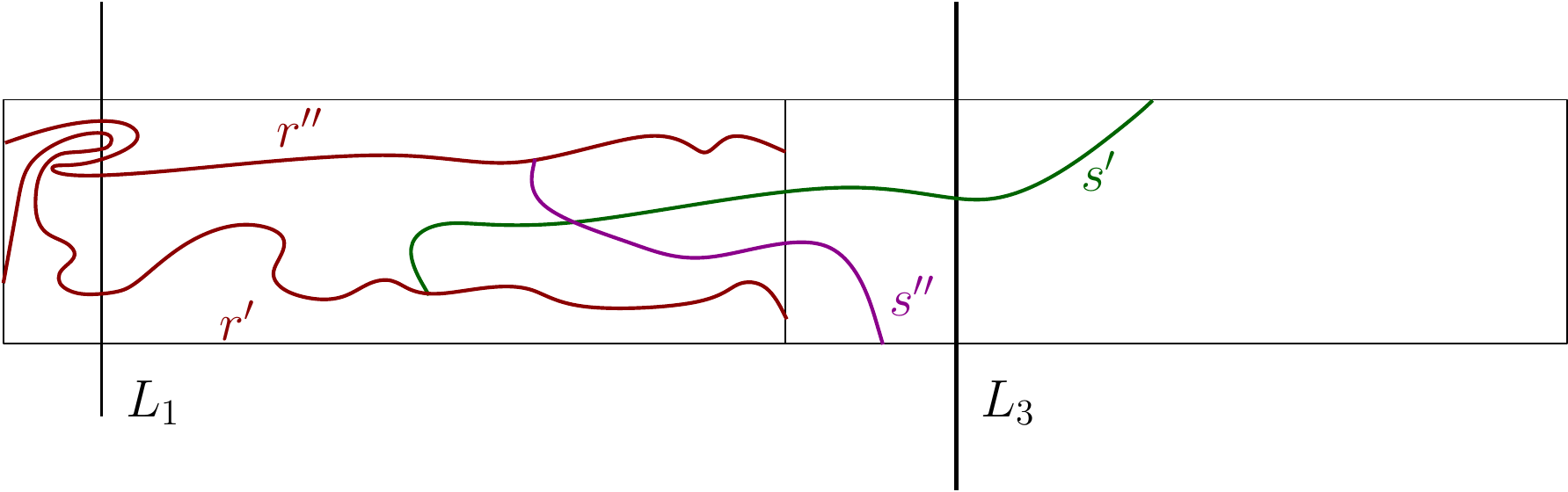} 
\caption{\it \(\psi \text{ contains one point on or to the left of }L_3\).}
\end{figure}

\item If \(\psi\) lies entirely to the left of \(L_3\), by bringing 
  in an occupied horizontal crossing \(r'''\) of
  \([\frac{l_1}{8},\frac{l_1}{8}+l_1]\times[0,l_2]\), then  \(r'\cup
  s'\cup r''\cup s''\cup r'''\) contains an occupied horizontal
  crossing of \([0,\frac{l_1}{8}+l_1]\times[0,l_2]\).

\begin{figure}[!h]
\label{Figure 15}
\centering
\includegraphics[width=.8\textwidth]{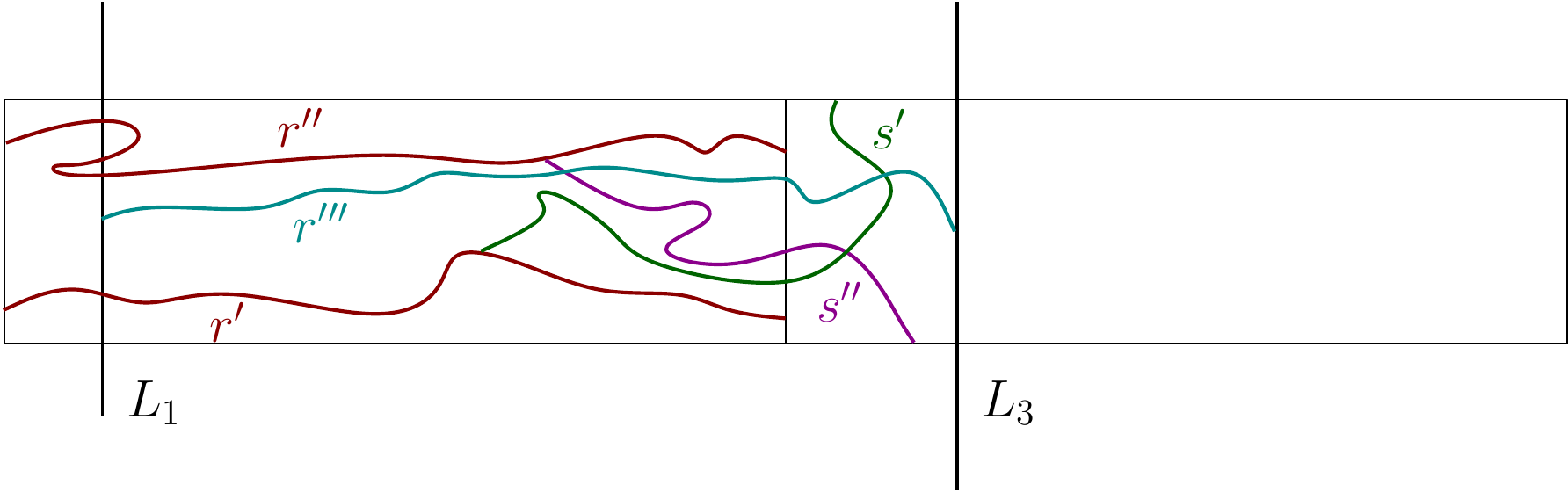} 
\caption{\it \(\psi\) lies to the left of \(L_3\), an occupied horizontal crossing \(r'''\) was introduced to ensure a crossing of \([0,\frac{l_1}{8}+l_1]\times[0,l_2]\)  .}
\end{figure}

\end{enumerate}  

By the FKG inequality, \eqref{eq:7} and \eqref{eq:8} we have 

\begin{equation}
  \label{eq:11}
  \P_p[\C_h(\frac{1}{8}l_1+l_1,l_2)]\geq (1-\sqrt{1-\delta_1})^2(1-\sqrt{1-\delta_2})^2\P_p[\C_h(l_1,l_2)].
\end{equation}

One can repeat this to obtain lower bounds for
\(\P_p[\C_h(\frac{jl_1}{8},l_2)]\) where \(j\geq 9\).
\end{proof}

It seems we are close to the result. Still, in order to complete the
proof of the main theorem, one should relax the condition \(l_3<2l_1\). In fact we
need the case \(l_3=4l_1\).

\subsection{Step four: Lemma \ref{lem2}}
The following lemma will serve to diminish \(l_3\), in other words, the lemma states that either we have a vertical crossing of \([0,\tilde{l_3}]\times [0,l_2]\) with \(\tilde{l_3}\) less than \(\frac{7}{4}l_1\), or we have a vertical and a horizontal crossing of \([0,\tilde{l_3}]\times[0,l_2]\) and \([0,\tilde{l_1}]\times[0,l_2]\)  with positive probability respectively, where the new \(\tilde{l_1},\tilde{l_3}\) satisfy the condition \(l_3\leq \frac{7}{4}l_1\) of Proposition \ref{kesten1}.

The idea is to divide the horizontal segment \([0,l_3]\) into pieces, if the width of the vertical crossing is not greater than a certain quantity, then the wished condition is satisfied, otherwise since the crossing is of long width, not only it provides a vertical crossing of certain box, but also it provides certain horizontal crossing of slightly smaller box, and using these two new crossings we can still apply Proposition \ref{kesten1}, see Figure \ref{fig:n}.

\begin{figure}[!h]
  \centering
  \includegraphics[width=.72\textwidth]{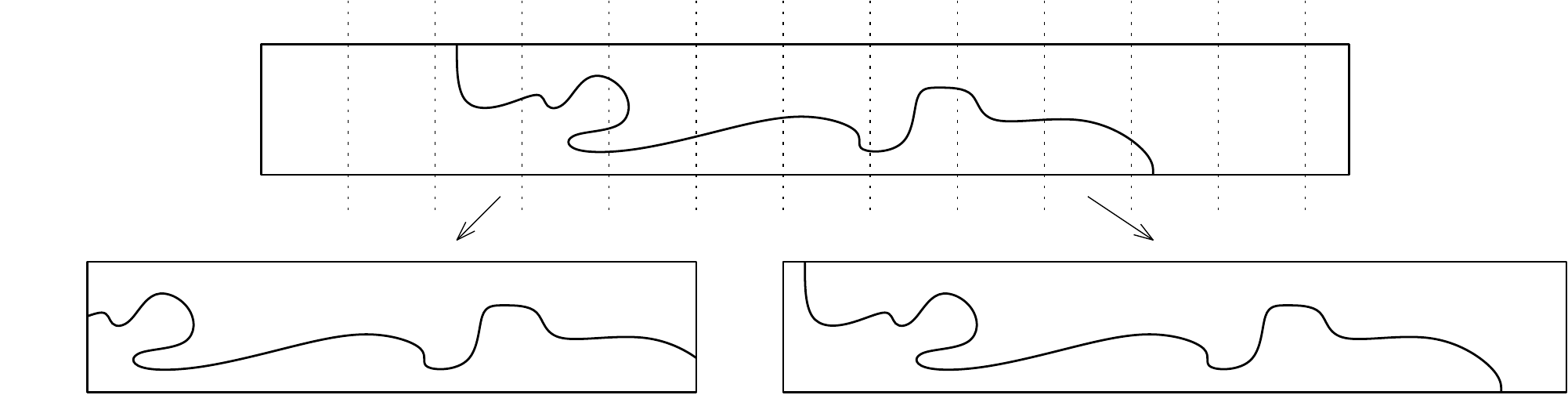}
  \caption{\it If the width of the vertical crossing is long, in fact it provides two crossings of appropriate size.}
  \label{fig:n}
\end{figure}

\begin{lemma}
\label{lem2}
Assume that \(\P_p[\C_v(l_3,l_2)]\geq \delta_2\), then there exists \(\delta_3>0\) such that one of the followings holds:
\begin{itemize}
\item \(\exists \tilde{l}_3\leq \frac{7}{4} l_1\), \(\P_p[\C_v(\tilde{l}_3,l_2)]\geq \delta_3\),
\item \(\exists \tilde{l}_1,\tilde{l}_3\) such that \(\tilde{l}_3\leq\frac{7}{4}\tilde{l}_1\) and \(\P_p[\C_h(\tilde{l}_1,l_2) \text{ and }\C_v(\tilde{l}_3,l_2)]\geq \delta_3\).
\end{itemize}
\end{lemma}

\begin{remark}
  This lemma does not depend on symmetry, the role of the
  horizontal and vertical direction may be interchanged. 
\end{remark}

\begin{proof}
  Let \(r\) be an occupied vertical crossing of
  \([0,l_3]\times[0,l_2]\). Denoted by \(x_l(r),x_r(r)\), the minimum and
  maximum value, respectively, of the first coordinates of sites 
  in \(r\). These are the most left and the most right site that \(r\) visited, see Figure \ref{Figure 16}.

 Let \(\lambda\) be a fixed constant, consider the subdivision of mesh \(\lambda l_1\) of segment \([0,l_3]\): 
\[I_k=\left \{ 
\begin{array}{ll}
\mbox{\([(k-1)\lambda l_1,k\lambda l_1)\)} &\mbox{ if \(1\leq k<n\) }\\
\mbox{\([(n-1)\lambda l_1,l_3]\)} &\mbox{ if \(k=n\). } 
\end{array} \right.\]
Where \(n=\lceil \frac{l_3}{\lambda l_1} \rceil\).

Let \(E(i,j)\) denotes the event: \\
{\it \{ there exists occupied vertical crossing \(r\) of \([0,l_3]\times[0,l_2]\) with \(x_l(r)\in I_i\) and \(x_r(r)\in I_j\). \} }

\begin{figure}[!h]
\centering
\includegraphics[width=\textwidth]{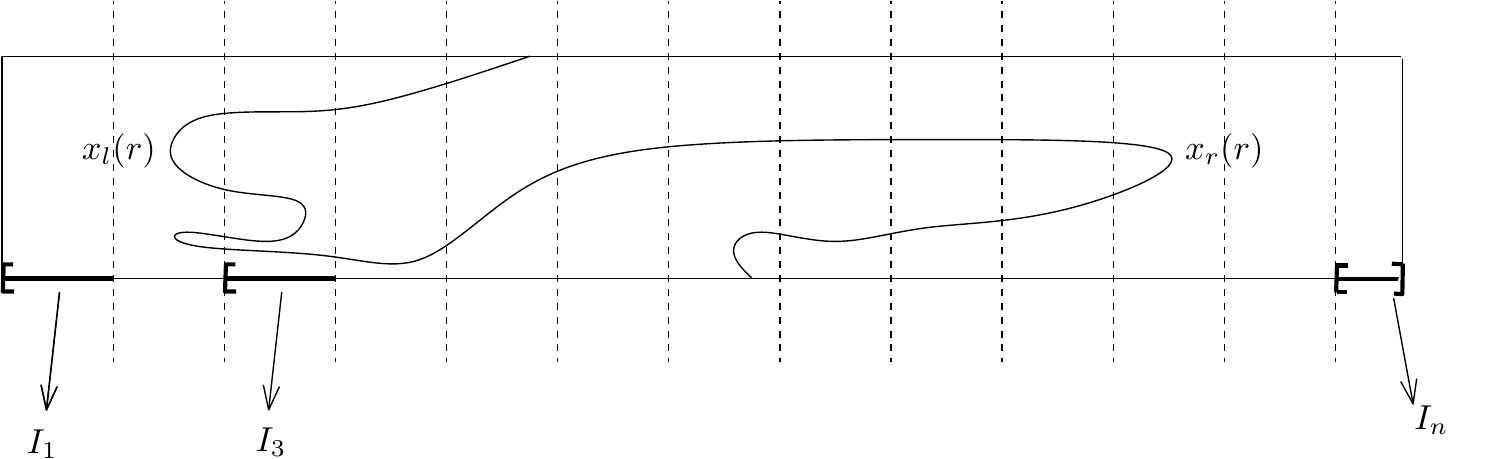}
\caption{\it \(x_l,x_r\) respectively the highest and lowest point of \(r\).}
\label{Figure 16}
\end{figure}

Any vertical crossing \(r\) of \([0,l_3]\times[0,l_2]\) must satisfy \(0\leq
x_l(r)\leq x_r(r)\leq l_3\), thus if there exists an occupied
vertical crossing of  \([0,l_3]\times[0,l_2]\) then one of the
events \(E(i,j), 1\leq i\leq j\leq n\) must occur. 

As the events \(E(i,j)\) are increasing, we can apply a more general form of the Harris
Inequality:  
\[ \prod_{k=1}^n(1-\P_p[E_k])\leq 1-\P_p[\cup_kE_k].\]
Therefore, 
\begin{equation}
  \label{eq:12}
  \prod_{1\leq i\leq j\leq n}(1-\P_p[E(i,j)])\leq 1-\P_p[\cup
  E(i,j)]\leq 1-\delta_1.
\end{equation}

The union and product in \eqref{eq:12} run over \(1\leq i\leq j\leq n\)
and hence contain \({n \choose 2}\) terms. Therefore, for some
\(1\leq i_0\leq j_0\leq n\) we have 
\begin{equation}
  \label{eq:13}
  \P_p[E(i_0,j_0)]\geq \delta_3=1-(1-\delta_1)^{{n\choose 2}^{-1}}.
\end{equation}

Now let us distinguish two case,  
\begin{enumerate}
\item If \( j_0-i_0\geq 4 \), then there exists an occupied vertical
  crossing \(r\) of \([0,l_3]\times[0,l_2]\) with 
\[ x_r(r)-x_l(r)\leq (j_0-i_0+1)\lambda l_1.\]

Thus by periodicity, \(r\) can be considered as an occupied vertical crossing of \[[0,(j_0-i_0+1)\lambda l_1]\times[0,l_2].\]

Next \(x_l(r)\leq i_0l_1< (j_0-1)l_1 \leq x_r(r)\) implies that there is an
occupied horizontal crossing of \([0,(j_0-i_0-1)\lambda l_1]\times[0,l_2 ]\).

Thus  \(\P_p[\C_v((j_0-i_0+1)\lambda l_1,l_2 )\text{ and } \C_h((j_0-i_0-1)\lambda l_1,l_2)] \geq \delta_3
\).

Let \(\tilde{l}_1=(j_0-i_0-1)\lambda l_1\) and \(\tilde{l}_3=(j_0-i_0+1)\lambda l_1\), now \(\frac{7}{4}\tilde{l}_1 \geq \tilde{l}_3\) thus the desired condition is satisfied.
\item If  \(j_0-i_0\leq 3\), then the first part of the above
  argument shows that \[\P_p[\C_v((j_0-i_0+1)\lambda l_1,l_2)]\geq \delta_3,\] and hence \(\tilde{l}_3=(j_0-i_0+1)\lambda l_1,\) let \(\lambda\leq \frac{7}{16}\) leads to the conclusion.
\end{enumerate}

\end{proof}

\subsection{Proof of Theorem  \ref{kesten2}}

\begin{proof}

By Proposition \ref{kesten1}, we may assume \(l_3\geq \frac{7}{4}l_1\). 
Now by Lemma \ref{lem2}, one of the following statements holds: 
\begin{itemize}
\item there exists \(\tilde{l}_3\leq \frac{7}{4}l_1\), such that \(\P_p[\C_v(\tilde{l}_3,l_2)]\geq \delta_3\).
\item there exists \(\tilde{l}_3,\tilde{l}_1\) and \(\tilde{l}_3\leq \frac{7}{4}\tilde{l}_1\) such that \(\P_p[\C_h(\tilde{l}_1,l_2)]\geq \delta_3\) and \(\P_p[\C_v(\tilde{l}_3,l_2)]\geq \delta_3\).
\end{itemize}

In either case we can apply Proposition \ref{kesten1}, hence  \[ \P_p[\C_h(kl_1,l_2)]\geq c>0.\]

\end{proof}

\section{Consequences}

Now the RSW Theorem is proved by combining Theorem \ref{kesten2} and  Theorem\ref{lem0}. 
Here we give a result on \(\Z^2\) site percolation model  which can be deduced from the RSW Theorem.
\subsection{No infinite cluster at criticality.}
Consider site percolation on \(\Z^2\), it is known that the critical value of \(p\), denoted \(p_c\), is well defined and \(p_c>\frac{1}{2}\).

A classic corollary of the RSW Theorem is 
\begin{proposition}
  There is no infinite cluster at criticality for site percolation on \(\Z^2\).
\end{proposition}

\begin{proof}

By periodicity, it is equivalent to say that the origin, denoted \(0\), is not connected to infinity a.s.

Note that the existence of a vacant circuit in \(\Z^{2,*}\) surrounding the origin prohibit \(0\) from connecting to the outside of this circuit. Consider the square annulus as shown in Figure \ref{Figure 17}.

\begin{figure}[!h]
\centering
\includegraphics[width=.8\textwidth]{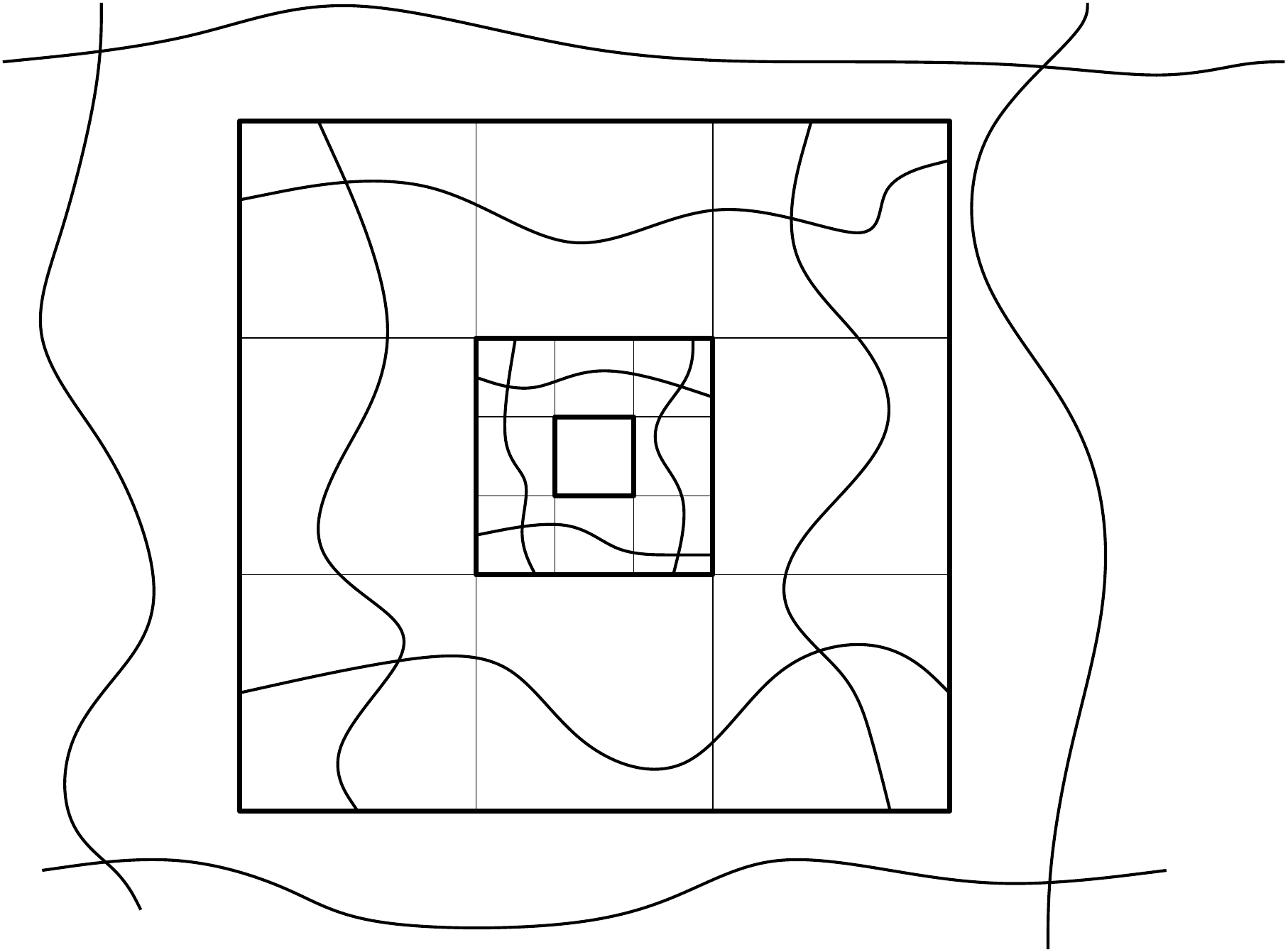}
\caption{\it In each annulus \(A_n\), we can construct 4 crossings that surround the origin.}
\label{Figure 17}
\end{figure}

In every annulus we construct four crossings of rectangles of size \(n\times 3n\), these 4 crossings must form a circuit surrounding the origin, and these events of circuit are independent in different annulus.

By the FKG Inequality, 
\begin{align*}
&\P[ \text{ there is a vacant dual circuit in annulus }A_n]\\
&\geq \P[\text{ there is a vacant dual crossing in the rectangle of size } n\times 3n]^4 \\
&\geq c^4.
\end{align*}

Now as \(\sum_n \P[ \text{ there is a vacant dual circuit in annulus }A_n] =\infty\), apply the Borel-Cantelli Lemma to conclude that there exists a vacant dual circuit surrounding \(0\) a.s.
\end{proof}

{\bf Acknowledgments:}

I would like to thank Christophe Garban for his constant support in this project as well as
Hugo Duminil-Copin who suggested to us the strategy to prove our main Theorem. I am also indebted to Christophe Sabot for a very careful reading of this manuscript.

\nocite{*}
\bibliography{bib}{}
\bibliographystyle{plain}
\end{document}